\documentclass[12pt]{article}
\usepackage{pdfsync}
\usepackage{graphicx}
\usepackage{amsmath, amssymb, amsthm}
\usepackage{geometry}
\usepackage{epstopdf}
\usepackage{amsfonts}
\usepackage{mathrsfs}
\usepackage{marvosym}
\usepackage[pdftex]{color}
\usepackage{pdfcolmk}
\usepackage{mathbbol}

\usepackage[dvips]{epsfig} \DeclareGraphicsExtensions{.ps,.eps}
\def\R{\mathbb{R}}

\def\1{\mathbb{1}}

\def\diag{{\rm diag}}

\newcommand{\tp}{^{\top}}

\newcommand{\beq}{\begin{equation}}
\newcommand{\eeq}{\end{equation}}
\newcommand{\bea}{\begin{eqnarray}}
\newcommand{\eea}{\end{eqnarray}}

\newcommand{\bsea}{\begin{subeqnarray}}
\newcommand{\esea}{\end{subeqnarray}}
\newcommand{\nn}{\nonumber}

\def\bmat{\left[ \begin{array}}
\def\emat{\end{array} \right]}

\newcounter{acount}

\newtheorem{theorem}{Theorem}[section]
\newtheorem{proposition}{Proposition}[section]

\newtheorem{lemma}{Lemma}[section]

\title{On the Positivity of a  Class of Cauchy-Like Matrices\thanks{This work was partially supported by the Italian Ministry of University and Research under the PRIN project ``Extracting essential information and dynamics from complex networks'', grant  no. 2022MBC2EZ.}
}
  
\author{Augusto Ferrante\thanks{Dipartimento di Ingegneria dell'Informazione, Universit\`a di Padova, Via Gradenigo 6/B, 35131 Padova. e-mail: {\tt augusto.ferrante@unipd.it}}}

\begin{document}
\maketitle

\begin{abstract}
Let $0<\lambda_1<\cdots<\lambda_n$.
Motivated by a problem related to Lyapunov equations we consider a class of 
Cauchy-like matrices whose elements have the form 
$C_{ij}=\frac{r_i(k,l)+r_j(k,l)}{\lambda_i+\lambda_j},$ where
for any pair $1\le k,l\le n$, $r_i(k,l)$ are functions of $\{\lambda_1,\cdots,\lambda_n\}\setminus \lambda_i$.
We show that
these matrices
are positive semidefinite for every pair $1\le k,l\le n$.  
After passing to the
reciprocal variables $x_i=1/\lambda_i$, the problem is reduced by a diagonal
congruence to the positivity of a two-parameter family
$A_n^{(p,q)}(x)$.  The proof introduces a singular augmented matrix
$\mathcal H_n^{(p,q)}(x)$, proves its singularity by Cauchy-kernel generating
function identities, and then proves positive semidefiniteness by induction on
$n$ using the principal-minor criterion.  
\end{abstract}


\section{Introduction}
In this paper we consider a class of Cauchy-like matrices defined as follows.
Let 
\beq\label{orderlam}
0<\lambda_1<\lambda_2< \dots <\lambda_n
\eeq
be {\em distinct}, {\em positive} real numbers.
Denote by $\sigma_{m}$ the \emph{elementary symmetric polynomial of degree $m$}
in the variables $\{\lambda_1, \dots, \lambda_n\}$
and by
$\sigma_{m}^{(j)}$ the  elementary symmetric polynomial of degree $m$
in the variables $\{\lambda_1, \dots, \lambda_n\} \setminus \{\lambda_j\}$.
Let $l,k\in\{1,2,\dots,n\}$ and define the following positive numbers:
\beq\label{rikl}
r_i(k,l):=\frac{\sigma_{n-k}^{(i)}}{\sigma_{n-l}^{(i)}}.
\eeq
Consider the class of Cauchy-like matrices  $C\in\R^{n\times n}$
whose elements are 
\beq\label{defcij}
C_{ij}=\frac{r_i(k,l)+r_j(k,l)}{\lambda_i+\lambda_j},\quad i,j=1,\dots,n.
\eeq

Matrices of this form are connected to a very well studied class of matrices:
the so-called Kwong (or anti-Loewner) matrices.
The latter are matrices $W$ whose elements have the form
$$
W_{ij}=\frac{f(\lambda_i)+f(\lambda_j)}{\lambda_i+\lambda_j}.
$$
The class of matrices under scrutiny here can be viewed as dual to Kwong matrices.
In fact, setting $\Lambda^{(i)}:= \{\lambda_j: \ j=1,\dots,n, \ j\neq i\}$,
we can write 
$C_{ij}=\frac{f(\Lambda^{(i)})+f(\Lambda^{(j)})}{\lambda_i+\lambda_j}$
where $f$ is a real-valued function. 

As a variation on the classical 
results of Loewner
\cite{Loewner1934}, a quite extensive literature has analyzed Kwong matrices and the properties of $f$ guaranteeing their positivity \cite{Kwong1989,BhatiaSano2009,Audenaert2011,Morishita2014,BhatiaJain2023}.
On the other hand, the matrices of the form discussed here have been so far neglected by the literature to the best of my knowledge. This paper may be viewed as a first attempt to start to fill this gap.

My main motivation comes from the analysis of the solution of a Lyapunov equation where the state matrix is in companion form, \cite{Ferrante-in-prep}.

The main result of this paper is the following.
\begin{theorem}
Let $\lambda_i$, $i=1,\dots,n$ be defined by (\ref{orderlam}) and 
$r_i(k,l)$ be defined by (\ref{rikl}). Then the matrix $C\in\R^{n\times n}$ whose elements are defined by (\ref{defcij}) is positive semidefinite for all $l,k\in\{1,2,\dots,n\}$.
\end{theorem}

\section{Preliminaries}
\subsection{Cauchy matrices}
We recall, see e.g, \cite{Bhatia2006}, that the Cauchy matrix 
${\mathcal C}$ whose elements are 
$$
{\mathcal C}_{ij}=\frac{1}{\lambda_i+\lambda_j},\quad i,j=1,\dots,n.
$$ is positive definite whenever the nodes $\lambda_i$ are all positive 
and pairwise distinct which is guaranteed by our assumptions.

\subsection{A result on rational functions}

\begin{lemma}\label{lemmaratfunc}
Let $\lambda_1,\dots,\lambda_n$ be distinct positive real numbers and \(g(s)\) be a 
strictly proper rational function having only simple poles at $-\lambda_i$, $i=1,\dots,n$ and no other poles.
Then \(g\) is uniquely determined by the \(n\) values \(g(\lambda_i),\dots,g(\lambda_n)\).
\end{lemma}

\begin{proof}
Let
\[
g_1(s)=\frac{Q_1(s)}{P(s)},\qquad g_2(s)=\frac{Q_2(s)}{P(s)}
\]
where \(P(s)=\prod_{k=1}^n(s+\lambda_k)\) and \(Q_1(s)\) and $Q_2(s)$ are  polynomials of degree at most \(n-1\). Assume that $g_1(\lambda_i)=g_2(\lambda_i)$, $i=1,\dots,n$.  Then
\[
Q_1(\lambda_i)=g_1(\lambda_i)P(\lambda_i)=g_2(\lambda_i)P(\lambda_i)=Q_2(\lambda_i),
\qquad i=1,\dots,n.
\]
Thus the polynomial \(Q_1(s)-Q_2(s)\), whose degree is at most \(n-1\), vanishes at the \(n\) distinct points \(\lambda_1,\dots,\lambda_n\). Hence \(Q_1(s)-Q_2(s)\equiv 0\) and therefore \(g_1(s)=g_2(s)\).
\end{proof}

\subsection{Problem reformulation}

Set
\beq\label{defxiX}
x_i:=\frac1{\lambda_i},
\qquad
X:=\diag(x_1,\dots,x_n).
\eeq
Thus the \(x_i\)'s are positive and pairwise distinct.

For a vector \(x=(x_1,\dots,x_m)\), let \(e_r(x)\) denote the elementary
symmetric polynomial of degree \(r\) in \(x_1,\dots,x_m\). We use the convention
\[
e_0(x)=1,\qquad e_r(x)=0 \quad \text{if } r<0 \text{ or } r>m.
\]
Also, \(e_r^{(i)}(x)\) denotes the elementary symmetric polynomial of degree
\(r\) in all variables except \(x_i\), with the analogous convention.
Notice that
$\frac{\sigma_n}{\lambda_i}=\prod_{j\neq i}\lambda_j$ so that
we have $\sigma_{n-m}^{(i)}
=
\frac{\sigma_n}{\lambda_i}e_{m-1}^{(i)}(x)$, for all $m=0,1,\dots, n$.
In particular, for $m=k$ and $m=l$, we obtain
\[
\sigma_{n-k}^{(i)}
=
\frac{\sigma_n}{\lambda_i}e_{k-1}^{(i)}(x),
\qquad
\sigma_{n-l}^{(i)}
=
\frac{\sigma_n}{\lambda_i}e_{l-1}^{(i)}(x).
\]
Hence
\[
r_i(k,l)
=
\frac{\sigma_{n-k}^{(i)}(\lambda)}
{\sigma_{n-l}^{(i)}(\lambda)}
=
\frac{e_{k-1}^{(i)}(x)}{e_{l-1}^{(i)}(x)}.
\]
Moreover, we have
\[
\lambda_i+\lambda_j
=
\frac1{x_i}+\frac1{x_j}
=
\frac{x_i+x_j}{x_ix_j}.
\]

Thus we can rewrite the elements of $C$ as
\[
C_{ij}
=
x_ix_j
\frac{\frac{e_{k-1}^{(i)}(x)}{e_{l-1}^{(i)}(x)}+\frac{e_{k-1}^{(j)}(x)}{e_{l-1}^{(j)}(x)}}{x_i+x_j}=
\frac{x_ix_j}{e_{l-1}^{(i)}(x)e_{l-1}^{(j)}(x)}
\frac{e_{k-1}^{(i)}(x) e_{l-1}^{(j)}(x)+e_{k-1}^{(j)}(x) e_{l-1}^{(i)}(x)}{x_i+x_j}.
\]

By setting $D_{l-1}:=
\diag\bigl(e_{l-1}^{(1)}(x),\dots,e_{l-1}^{(n)}(x)\bigr)$,
we get
\[
C
=
X D_{l-1}^{-1}
A_n^{(k-1,l-1)}(x)
D_{l-1}^{-1}X
\]
where, for \(1\le  k,l \le n\), $A_n^{(k-1,l-1)}(x)$ is the 
$n\times n$ matrix whose elements are
\[
[A_n^{(k-1,l-1)}(x)]_{ij}
=
\frac{
e_{k-1}^{(i)}(x)e_{l-1}^{(j)}(x)
+
e_{k-1}^{(j)}(x)e_{l-1}^{(i)}(x)
}
{x_i+x_j}.
\]
Since $X$ and $D_{l-1}$ are nonsingular diagonal (and hence symmetric) matrices, by setting
$p:=k-1$ and $q:=l-1$, our problem is equivalent to
\beq
A_n^{(p,q)}(x)\succeq0
\qquad
\text{for all }0\le p,q\le n-1.
\eeq
Notice that
$
A_n^{(p,q)}(x)=A_n^{(q,p)}(x).$

\section{Proof of the main result}

\subsection{Partial analysis of the case $l=n$}

For \(l=n\), the elements of the matrix $A_n^{(k-1,l-1)}(x)$ reduce to the form 
$$
\begin{aligned}
[A_n^{(k-1,n-1)}(x)]_{ij}
&=
\frac{
e_{k-1}^{(i)}(x)e_{n-1}^{(j)}(x)
+
e_{k-1}^{(j)}(x)e_{n-1}^{(i)}(x)
}
{x_i+x_j}\\
&=
\frac{
e_{k-1}^{(i)}(x)e_{n}(x)/x_j
+
e_{k-1}^{(j)}(x)e_{n}(x)/x_i
}
{x_i+x_j}\\
&=
e_{n}(x) 
\frac{1}{x_ix_j}
\frac{
e_{k-1}^{(i)}(x) x_i
+
e_{k-1}^{(j)}(x) x_j
}
{x_i+x_j}
\end{aligned}
$$

As a consequence, 
$$
A_n^{(k-1,n-1)}(x)=e_{n}(x) X^{-1} B_n^{(k-1)}(x)X^{-1}
$$
with $B_n^{(k-1)}(x)$ being the $n\times n$ matrix having elements
\[
\bigl[B_n^{(k-1)}(x)\bigr]_{ij}
:=\frac{x_i e_{k-1}^{(i)}(x)+x_j e_{k-1}^{(j)}(x)}{x_i+x_j}.
\]

Since \(e_{n}(x) >0\) and \(X=X\tp\) is invertible, 
in the particular case of $l=n$, our result is equivalent to 
\[
B_n^{(r)}(x)\succeq 0
\qquad
\text{for }0\le r\le n-1.
\]

The matrix $B_n^{(r)}(x)$ indeed plays an important role also for the general case.
Thus, it is interesting to derive some results on such a matrix. Define
 $K$ to be the Cauchy matrix associated with the nodes $x_i$:
\beq\label{defofK}
K:=\left(\frac1{x_i+x_j}\right)_{i,j=1}^n.
\eeq
We also set
\beq\label{defureUr}
u_r:=
\begin{bmatrix}
e_r^{(1)}(x)\\
\vdots\\
e_r^{(n)}(x)
\end{bmatrix},
\qquad
U_r:=\diag(u_r),\qquad 0\le r\le n-1,
\eeq
Then,
we can rewrite $B_n^{(r)}(x)$ as
\beq\label{altraespdib}
B_n^{(r)}(x)=XU_rK+KU_rX.
\eeq

We will use the following result.
\begin{proposition}\label{prop:cauchy-kernel-identities}
Let \(x_1,\dots,x_n>0\) be pairwise distinct, and define $X$, $K$ 
$u_r$ and $U_r$ be defined by (\ref{defxiX}), (\ref{defofK})
and (\ref{defureUr}).
Let 
$$h:=
\begin{bmatrix}
1/x_1\\
\vdots\\
1/x_n
\end{bmatrix}=X^{-1}\1,\qquad
z:=\frac12K^{-1}h, \qquad w:=Xz.
$$
Finally, let 
\[
\alpha_r:=\frac{1-(-1)^{n+r}}2,\qquad \beta_r:=\alpha_r-\frac12
=
\frac{(-1)^{n+r+1}}2.
\]

Then, for every \(0\le r\le n-1\), we have:
\beq\label{primofatto}
 u_r^\top z=\alpha_r e_r(x),
\eeq
\beq\label{secondofatto}
B_n^{(r)}(x)z=\alpha_r u_r,
\eeq
and
\beq\label{terzofatto}
KU_rw=\beta_r u_r.
\eeq
\end{proposition}

\begin{proof}
First, \(K\) is a Cauchy matrix associated with positive distinct nodes and hence it is positive  definite. Thus \(K\) is invertible, and \(z\) is
well-defined.

By definition of \(z\),
\[
Kz=\frac12h,
\]
and, denoting by $z_j$ the $j-$th component of $z$, this is equivalent to
\beq\label{rdizj}
\sum_{j=1}^n \frac{z_j}{x_i+x_j}
=
\frac1{2x_i},
\qquad i=1,\dots,n.
\eeq

Define
\[
P(s):=\prod_{i=1}^n(s+x_i),
\qquad
\rho(s):=\frac{P(-s)}{P(s)},\qquad R(s):=\sum_{j=1}^n\frac{z_j}{s+x_j}.
\]

By taking (\ref{rdizj}) into account, we have
\[
R(x_i)=\frac1{2x_i},\qquad \rho(x_i)=0,
\qquad i=1,\dots,n.
\]

We claim that
\beq\label{r=umrs2s}
R(s)=\frac{1-\rho(s)}{2s}.
\eeq
Indeed, both sides of this equality are strictly proper rational functions 
having only simple poles at
\(-x_1,\dots,-x_n\). In fact, \(\rho(0)=1\) so that the singularity 
at \(s=0\) of $\frac{1-\rho(s)}{2s}$ is removable.

Moreover, since \(P(-x_i)=0\) for every \(i=1,\dots,n\), we have
\[
\frac{1-\rho(x_i)}{2x_i}
=
\frac1{2x_i},\qquad i=1,\dots,n.
\]
Hence (\ref{r=umrs2s}) follows from Lemma \ref{lemmaratfunc}.

Now define
\[
\tau:=\frac1{s}.
\]
and let
\beq\label{defdiE}
E(\tau):=\tau^n P(1/\tau)=\prod_{i=1}^n(1+x_i\tau)
=
\sum_{r=0}^n e_r(x)\tau^r,
\eeq
and, for \(i=1,\dots,n\),
\beq\label{defdiUi}
U_i(\tau):=\frac{E(\tau)}{1+x_i\tau}
=
\sum_{r=0}^{n-1} e_r^{(i)}(x)\tau^r.
\eeq

To prove our result, we derive some relations linking $E(\tau)$ with $R(s)$ and $\rho(s)$ (remember that $s=1/\tau$).

To prove (\ref{primofatto}), we have 
\begin{align}\nn
\sum_{r=0}^{n-1} u_r^\top z\,\tau^r
&=\sum_{r=0}^{n-1} [e_r^{(1)}(x)\ \dots\ e_r^{(n)}(x)] \begin{bmatrix}z_1\\
\vdots\\
z_n\end{bmatrix}\tau^r=\sum_{r=0}^{n-1}
\sum_{i=1}^n e_r^{(i)}(x) z_i\tau^r=\sum_{i=1}^n z_i
\underbrace{\sum_{r=0}^{n-1}
 e_r^{(i)}(x) \tau^r}_{U_i(\tau)}\\
 \label{exprofsum}
&=E(\tau)\sum_{i=1}^n\frac{z_i}{1+x_i\tau}.
\end{align}
Now observe that
\beq\label{eqforsumoffra}
\begin{aligned}
\sum_{i=1}^n\frac{z_i}{1+x_i\tau}&=
\sum_{i=1}^n\frac{sz_i}{s+x_i}
=
sR(s)=
s\frac{1-\rho(s)}{2s}=
(1/2) \left[1-\frac{P(-1/\tau)}{P(1/\tau)}\right]\\
&= (1/2) \left[\frac{\tau^n P(1/\tau) - \tau^n P(-1/\tau)}{\tau^n P(1/\tau)}\right]=\frac12\frac{E(\tau)-(-1)^nE(-\tau)}{E(\tau)}.
\end{aligned}
\eeq
By plugging (\ref{eqforsumoffra}) into (\ref{exprofsum}) and taking
the expansion (\ref{defdiE}) into account, we get
\[
\sum_{r=0}^{n-1} u_r^\top z\,\tau^r
=\frac12\left[E(\tau)-(-1)^nE(-\tau)\right]=
\frac12 \sum_{r=0}^n e_r(x)\tau^r\underbrace{(1-(-1)^{n+r})}_{2\alpha_r}
\]
By comparing the coefficients of \(\tau^r\), we now get (\ref{primofatto}).

We now prove (\ref{secondofatto}).
Fix \(i\in\{1,\dots,n\}\). Using the generating function \(U_i(\tau)\), we have
\[
\begin{aligned}
\sum_{r=0}^{n-1}
\bigl[B_n^{(r)}(x)z\bigr]_i\tau^r
&=\sum_{r=0}^{n-1}
\sum_{j=1}^n [B_n^{(r)}(x)]_{ij} z_j\tau^r=
\sum_{j=1}^n z_j
\sum_{r=0}^{n-1}[B_n^{(r)}(x)]_{ij} \tau^r\\
&=
\sum_{j=1}^n z_j
\sum_{r=0}^{n-1}\frac{x_i e_r^{(i)}+x_j e_r^{(j)}}{x_i+x_j} \tau^r=
\sum_{j=1}^n z_j \frac{x_iU_i(\tau)+x_jU_j(\tau)}{x_i+x_j}\\
&=
\sum_{j=1}^n z_j \frac{x_iU_i(\tau)}{x_i+x_j} +z_jx_j\frac{E(\tau)}{(x_i+x_j)(1+x_j\tau)}\\
&=
x_iU_i(\tau)\sum_{j=1}^n\frac{z_j}{x_i+x_j}
+
E(\tau)
\sum_{j=1}^n
z_j\frac{x_j}{(1+x_j\tau)(x_i+x_j)}\\
&=
x_iU_i(\tau)\frac1{2x_i}
+
E(\tau)
\sum_{j=1}^n
z_j\frac{sx_j}{(s+x_j)(x_i+x_j)}\\
&=
\frac12U_i(\tau)
+
E(\tau)
\sum_{j=1}^n
z_j\left[\frac{s^2}{s-x_i}\frac1{s+x_j}
-
\frac{s x_i}{s-x_i}\frac1{x_i+x_j}\right]\\
&=
\frac12U_i(\tau)
+
E(\tau)
\left[
\frac{s^2}{s-x_i}\sum_{j=1}^n\frac{z_j}{s+x_j}
-
\frac{s x_i}{s-x_i}\sum_{j=1}^n\frac{z_j}{x_i+x_j}
\right]\\
&=
\frac12U_i(\tau)
+
E(\tau)
\left[
\frac{s^2R(s)}{s-x_i}
-
\frac{s x_i}{s-x_i}\frac1{2x_i}
\right]\\
&=
\frac12U_i(\tau)
+
E(\tau)
\frac{s}{s-x_i}
\left(sR(s)-\frac12\right)\\
&=
\frac12U_i(\tau)
+
E(\tau)
\frac{s}{s-x_i}
\left(\frac{1-\rho(s)}2-\frac12\right)\\
&=
\frac12U_i(\tau)
+
E(\tau)
\frac{s}{s-x_i}
\left(-\frac{\rho(s)}2\right)\\
&=
\frac12\left[ U_i(\tau)
-
\frac{s}{s-x_i}E(\tau)\rho(s)\right]\\
&=
\frac12\left[U_i(\tau)
-
\frac{1}{1-x_i\tau}E(\tau)
\frac{P(-1/\tau)\tau^n}{P(1/\tau)\tau^n}\right]\\
&=
\frac12\left[U_i(\tau)
-
\frac{1}{1-x_i\tau}E(\tau)
\frac{E(-\tau)(-1)^n}{E(\tau)}\right]\\
&=
\frac12\left[U_i(\tau)
-
(-1)^n \frac{E(-\tau)}{1-x_i\tau}\right]\\
&=
\frac12\left[U_i(\tau)
-
(-1)^n U_i(-\tau)\right]\\
&=
\frac12\left[\sum_{r=0}^{n-1}e_r^{(i)}(x)\tau^r
-
\sum_{r=0}^{n-1} (-1)^{n} e_r^{(i)}(x)(-\tau)^r\right]\\
&=
\sum_{r=0}^{n-1}\frac{1-(-1)^{n+r}}{2} e_r^{(i)}(x)\tau^r.
\end{aligned}
\]
Comparing the coefficients of $\tau^r$ gives
\[
\bigl[B_n^{(r)}(x)z\bigr]_i
=
\frac{1-(-1)^{n+r}}2 e_r^{(i)}(x)
=
\alpha_r e_r^{(i)}(x),\qquad i=1,\dots,n
\]
which is equivalent to (\ref{secondofatto}).

Finally, taking into account (\ref{altraespdib}), the definitions of $h$ of $z$ and of $w$ and (\ref{secondofatto}), we get
\[
\alpha_r u_r=B_n^{(r)}(x)z
=
XU_rKz+KU_rXz=
\frac12 XU_rh+KU_rw
=\frac12 U_r\1 +KU_rw=
\frac12u_r+KU_rw,
\]
so that
\[
KU_rw
=
\left(\alpha_r-\frac12\right)u_r
=
\beta_r u_r.
\]
\end{proof}

\subsection{A scalar positivity result}

For \(0\le p,q\le n-1\), set
\[
\mu:=\min\{p,q\},
\qquad
\nu:=\max\{p,q\},
\]
and define
\beq\label{defofs}
s_{p,q}(x)
:=
\sum_{a=0}^{\mu}
(-1)^{\mu-a}e_a(x)e_{\mu+\nu+1-a}(x).
\eeq
Thus \(s_{p,q}=s_{q,p}\). For example,
$
s_{0,q}=e_{q+1},
$
and
$
s_{1,q}=e_1e_{q+1}-e_{q+2}.
$

We shall need the following elementary positivity fact.

\begin{lemma}
For \(x_1,\dots,x_n>0\),
\[
s_{p,q}(x)\ge0
\qquad
\text{ for all }0\le p,q\le n-1.
\]
\end{lemma}

\begin{proof}
Assume without loss of generality that \(p\le q\). Then
\beq\label{defscalare}
s_{p,q}(x)
=
\sum_{a=0}^{p}
(-1)^{p-a}e_a(x)e_{p+q+1-a}(x).
\eeq
We expand this signed sum combinatorially. A monomial in
\((-1)^{p-a}e_a(x)e_{p+q+1-a}(x)\) is obtained by choosing two subsets
\[
A,B\subset\{1,\dots,n\},
\qquad
|A|=a,\quad |B|=p+q+1-a.
\]
Its sign is \((-1)^{p-a}\), and its absolute value is
$
x_Ax_B,
$
where for any subset $J$ of $\{1,\dots,n\}$, $x_J:=\prod_{j\in J} x_j$.
Of course, there may be many 
monomials in the sum (\ref{defscalare}) having the same absolute value but possibly different sign.
To parametrize all such monomials, consider the disjoint subsets
$$
L:= A\cap B,\qquad  \qquad  T:=(A\setminus L)\cup (B\setminus L).
$$
Then
$$x_Ax_B=x_L^2x_T$$
and $x_L^2x_T$ uniquely determines the disjoint sets $L$ and $T$.
Hence, to parametrize all the monomials in the sum (\ref{defscalare}) having the same absolute value $x_Ax_B$, we need to parametrize all
$A$ and $B$ corresponding to the same $L,T$.\
Fix two disjoint subsets \(L,T\subset\{1,\dots,n\}\), and consider the monomial
\[
x_L^2x_T.
\]
Let
$d:=|L|$ and observe that $d\leq p$ and $d\leq q$.
Since
$
|A|+|B|=p+q+1,
$
we must have
\[
|T|=p+q+1-2d.
\]
For a fixed pair \((L,T)\), the possible choices of \(A\) and \(B\) are obtained by
choosing a subset \(D\subset T\), and setting
\[
A=L\cup D,
\qquad
B=L\cup(T\setminus D).
\]
The condition \(|A|\le p\) gives
\[
|D|\le p-d.
\]
Therefore the total coefficient of \(x_L^2x_T\) in \(s_{p,q}(x)\) is
\[
\sum_{h=0}^{p-d}
(-1)^{p-d-h}
\binom{|T|}{h}=\sum_{h=0}^{p-d}
(-1)^{p-d-h}
\binom{p+q+1-2d}{h}
=
\binom{p+q-2d}{p-d}\geq 0
\]
where we have used 
the following elementary alternating-binomial identity holding for $0\leq m<N$:
\[
\sum_{h=0}^m(-1)^{m-h}\binom{N}{h}
=
\binom{N-1}{m}.
\]

Hence
$
s_{p,q}(x)$ is a sum of positive terms of the type $x_L^2x_T$ weighted by nonnegative coefficients and therefore is manifestly nonnegative.
Indeed, we can explicitly write $
s_{p,q}(x)$ as
\beq\label{posexspq}
s_{p,q}(x)
=
\sum_{d=0}^{\mu}
\binom{p+q-2d}{\mu-d}
\sum_{\substack{L,T\subset\{1,\dots,n\}\\
L\cap T=\varnothing\\
|L|=d,\ |T|=p+q+1-2d}}
x_L^2x_T,
\eeq
where the inner sum is zero if the cardinality conditions cannot be satisfied.

\end{proof}

\subsection{The augmented matrix}

For \(0\le r\le n-1\), let
$
u_r$ and $U_r:=\diag(u_r)$ be defined by (\ref{defureUr}) and, for \(0\le p,q\le n-1\), let $s_{p,q}(x)$ be given by (\ref{defofs}) and define
\[
v_{p,q}:=u_p\circ u_q,
\]
where \(\circ\) denotes entry-wise Hadamard product.

Consider the block matrix
\beq\label{defofcalH}
\mathcal H_n^{(p,q)}(x)
:=
\begin{bmatrix}
s_{p,q}(x) & v_{p,q}^\top\\
v_{p,q} & A_n^{(p,q)}(x)
\end{bmatrix}.
\eeq

We will show that $\mathcal H_n^{(p,q)}(x)\succeq0$ for all $0\le p,q\le n-1$ which clearly implies $A_n^{(p,q)}(x)\succeq0$ and thus our main result.
To this aim we first need to show that $\mathcal H_n^{(p,q)}(x)$ is  singular for all $0\le p,q\le n-1$.

\begin{lemma}\label{detcalH=0}
Let \(x_1,\dots,x_n>0\) be pairwise distinct. Then
\[
\det\mathcal H_n^{(p,q)}(x)=0,\qquad \forall \ 0\le p,q\le n-1.
\]
\end{lemma}
\begin{proof}
First of all notice that $A_n^{(p,q)}(x)$ can be rewritten as 
\[
A_n^{(p,q)}(x)=U_pKU_q+U_qKU_p,
\]
where $K$ is the Cauchy matrix defined in (\ref{defofK}) and 
the diagonal matrices $U_p$ and $U_q$ are defined in (\ref{defureUr}).
Consider the vector $w:=Xz= X\frac12K^{-1}h=X \frac12K^{-1}X^{-1}\1$ and observe that $w\neq 0.$ By taking (\ref{terzofatto}) into account, we get
\beq\label{Aw=gammav}
A_n^{(p,q)}(x)w
=U_pKU_qw+U_qKU_pw=U_p\beta_q u_q+U_q\beta_p u_p=
(\beta_p+\beta_q)v_{p,q}=\gamma_{p,q}v_{p,q}
\eeq
where we have defined 
\[
\gamma_{p,q}:=\beta_p+\beta_q
=
\frac{(-1)^{n+p+1}+(-1)^{n+q+1}}2.
\]

We also need the scalar identity
\beq\label{scalaridentity}
v_{p,q}^\top w=\gamma_{p,q}s_{p,q}(x).
\eeq
To prove it, we use 
$
E(\tau)$ and
$
U_i(\tau)$ defined in (\ref{defdiE}) and (\ref{defdiUi}),
respectively.
Indeed, we use the following bivariate generating function and we get 
\[
\begin{aligned}
\sum_{p,q=0}^{n-1}v_{p,q}^\top w\,\tau^p\eta^q
&=
\sum_{p,q=0}^{n-1} \sum_{i=1}^n e_p^{(i)}e_q^{(i)} x_iz_i\,\tau^p\eta^q\\
&=
\sum_{i=1}^n x_i z_i U_i(\tau)U_i(\eta)\\
&=
\sum_{i=1}^n x_i z_i \frac{E(\tau)}{1+x_i\tau}\frac{E(\eta)}{1+x_i\eta}\\
&=
E(\tau) E(\eta)
\sum_{i=1}^n z_i \frac1{\tau-\eta}
\left(
\frac1{1+x_i\eta}-\frac1{1+x_i\tau}
\right)\\
&=
\frac{E(\tau) E(\eta)}{\tau-\eta}
\sum_{i=1}^n 
\left(
\frac{z_i }{1+x_i\eta}-\frac{z_i }{1+x_i\tau}
\right)\\
&= \frac{E(\tau) E(\eta)}{2(\tau-\eta)}\left(\frac{E(\eta)-(-1)^nE(-\eta)}{E(\eta)}-\frac{E(\tau)-(-1)^nE(-\tau)}{E(\tau)}\right),
\end{aligned}
\]
where to obtain the last equality we have used twice (for $\eta$ and for $\tau$) the identity (\ref{eqforsumoffra}).
Developing the products we get
\beq\label{bivariateid}
\sum_{p,q=0}^{n-1}v_{p,q}^\top w\,\tau^p\eta^q
=
\frac{(-1)^n}{2(\tau-\eta)}
\left(
E(\eta)E(-\tau)-E(\tau)E(-\eta)
\right).
\eeq

Now, we expand the right-hand side of (\ref{bivariateid})
to compare the coefficient of \(\tau^p\eta^q\) with the corresponding coefficient of the left-hand side, i.e.  \(v_{p,q}^{\top}w\).
We assume that
\(p\le q\).  The case \(q<p\) follows by symmetry.
From (\ref{defdiE}) it follows that
$
E(\eta)E(-\tau)
=
\sum_{a=0}^n\sum_{b=0}^n
e_a(x)e_b(x)(-1)^b \eta^a\tau^b
$
and
$
E(\tau)E(-\eta)
=
\sum_{a=0}^n\sum_{b=0}^n
e_a(x)e_b(x)(-1)^b \tau^a\eta^b.
$
Hence,
\[
E(\eta)E(-\tau)-E(\tau)E(-\eta)
=
\sum_{a,b=0}^n
e_a(x)e_b(x)(-1)^b
\left(\eta^a\tau^b-\tau^a\eta^b\right).
\]

%Now group together the terms indexed by \((a,b)\) and \((b,a)\).  
The
terms with \(a=b\) vanish. 
Moreover,  if \(a_0<b_0 \), the combined contribution of the  
terms indexed by \((a,b)=(a_0,b_0)\) and by \((a,b)=(b_0,a_0)\)
is
\[
e_{a_0}(x)e_{b_0}(x)
\left((-1)^{b_0}-(-1)^{a_0}\right)
\left(\tau^{b_0}\eta^{a_0}-\tau^{a_0}\eta^{b_0}\right).
\]
Hence
\[
E(\eta)E(-\tau)-E(\tau)E(-\eta)
=
\sum_{0\le a<b\le n}
e_a(x)e_b(x)
\left((-1)^b-(-1)^a\right)
\left(\tau^b\eta^a-\tau^a\eta^b\right).
\]

We now divide by \(\tau-\eta\).  For \(a<b\),
\[
\frac{\tau^b\eta^a-\tau^a\eta^b}{\tau-\eta}
=
\tau^a\eta^a
\frac{\tau^{b-a}-\eta^{b-a}}{\tau-\eta}
=
\tau^a\eta^a\sum_{m=0}^{b-a-1}
\tau^{b-a-1-m}\eta^m
=\sum_{m=0}^{b-a-1}
\tau^{b-1-m}\eta^{a+m}.
\]
Thus
\[
\frac{E(\eta)E(-\tau)-E(\tau)E(-\eta)}{\tau -\eta}
=
\sum_{0\le a<b\le n}
e_a(x)e_b(x)
\left((-1)^b-(-1)^a\right)
\sum_{m=0}^{b-a-1}
\tau^{b-1-m}\eta^{a+m}.
\]

Thus the monomial \(\tau^p\eta^q\) appears in the contribution of the
pair \((a,b)\) precisely when there exists an integer \(m\in\{0,\dots, b-a-1\}\) such that
$
p=b-1-m,
$ and $
q=a+m.
$
Adding the two equations gives
$
p+q=a+b-1.
$
Equivalently,
$
b=p+q+1-a.
$
Since we are assuming \(p\le q\), \(0\leq a<b\), and $
0\le m\le b-a-1$, the possible values of \(a\) are exactly
\beq\label{valoripossdia}
a=0,1,\dots,p.
\eeq
Indeed, $q=a+m\leq a+b-a-1=b-1$ so that $a+b-1=p+q\leq p+b-1$ and hence
$a\leq p$.
On the other hand, if $a$ satisfies (\ref{valoripossdia})
and we set
$
b=p+q+1-a,
$
then \(a<b\), and the corresponding value
$
m=q-a
$
satisfies
$
0\le m\le b-a-1.
$

Therefore the coefficient of \(\tau^p\eta^q\) is
\[
v_{p,q}^{\top}w
=
\frac{(-1)^n}{2}
\sum_{a=0}^{p}
e_a(x)e_{p+q+1-a}(x)
\left(
(-1)^{p+q+1-a}-(-1)^a
\right).
\]

We now distinguish two parity cases.
First suppose that \(p+q\) is odd.  Then \(p+q+1\) is even.  Hence
\(p+q+1-a\) has the same parity as \(a\), and therefore
$
(-1)^{p+q+1-a}=(-1)^a.
$
Thus every summand in the preceding formula vanishes, and so
when \(p+q\) is odd
$
v_{p,q}^{\top}w=0.
$
On the other hand, when \(p+q\) is odd
$
\gamma_{p,q}
=
\frac{(-1)^{n+p+1}+(-1)^{n+q+1}}{2}=0.
$
Consequently, when \(p+q\) is odd
$
v_{p,q}^{\top}w
=
0
=
\gamma_{p,q}s_{p,q}(x)
$ and hence the desired scalar identity (\ref{scalaridentity})
holds.

Now suppose that \(p+q\) is even.  Then \(p+q+1\) is odd.  Hence
\(p+q+1-a\) has parity opposite to that of \(a\), and therefore
\[
(-1)^{p+q+1-a}
=
-(-1)^a.
\]
It follows that
\[
(-1)^{p+q+1-a}-(-1)^a
=
-2(-1)^a.
\]
Substituting this into the coefficient formula gives
\[
\begin{aligned}
v_{p,q}^{\top}w
&=
\frac{(-1)^n}{2}
\sum_{a=0}^{p}
e_a(x)e_{p+q+1-a}(x)
\left(-2(-1)^a\right) \\[1mm]
&=
(-1)^{n+1}
\sum_{a=0}^{p}
(-1)^a e_a(x)e_{p+q+1-a}(x)\\
&=(-1)^{n+p+1}
\sum_{a=0}^{p}
(-1)^{p-a}e_a(x)e_{p+q+1-a}(x).
\end{aligned}
\]

Since \(p+q\) is even, the integers \(p\) and \(q\) have the same parity.
Thus
$
(-1)^{n+p+1}
=
(-1)^{n+q+1},
$
and consequently
\[
\gamma_{p,q}
=
\frac{(-1)^{n+p+1}+(-1)^{n+q+1}}{2}
=
(-1)^{n+p+1}.
\]
Moreover, since \(p\le q\), the definition of \(s_{p,q}\) gives
\[
s_{p,q}(x)
=
\sum_{a=0}^{p}
(-1)^{p-a}e_a(x)e_{p+q+1-a}(x).
\]
Therefore
\[
v_{p,q}^{\top}w
=
\gamma_{p,q}s_{p,q}(x).
\]
This concludes the proof of the scalar identity (\ref{scalaridentity}) for \(p\le q\).  Since the definitions of
\(v_{p,q}\), \(A_n^{(p,q)}\), \(\gamma_{p,q}\), and \(s_{p,q}\) are symmetric
in \(p\) and \(q\), the same identity follows for \(q<p\).

Combining (\ref{Aw=gammav}) and (\ref{scalaridentity}) we immediately get
\beq\label{kerofcalH}
\mathcal H_n^{(p,q)}(x)
\begin{bmatrix}
-\gamma_{p,q}\\
w
\end{bmatrix}
=0
\eeq
and, since $w\neq 0$, this implies
\beq\label{detH=0}
\det\mathcal H_n^{(p,q)}(x)=0.
\eeq
\end{proof}

We are now ready for our last result
\begin{proposition}
Let \(x_1,\dots,x_n>0\) be pairwise distinct and  
$\mathcal H_n^{(p,q)}(x)$ be defined by (\ref{defofcalH}).
Then
\beq\label{result-extended}
\mathcal H_n^{(p,q)}(x)\succeq 0,\qquad \forall \ 0\le p,q\le n-1.
\eeq
\end{proposition}
\begin{proof}
We prove the proposition by induction on \(n\).
For \(n=1\), the only possible pair is \(p=q=0\). We have
\[
A_1^{(0,0)}(x)=\left[\frac1{x_1}\right],
\qquad
s_{0,0}(x)=x_1,
\qquad
v_{0,0}=[1],
\]
and hence
\[
\mathcal H_1^{(0,0)}(x)
=
\begin{bmatrix}
x_1&1\\
1&1/x_1
\end{bmatrix}
\succeq0.
\]

Assume the proposition has been proved in all dimensions strictly smaller than
\(n\). Fix \(0\le p,q\le n-1\). 
In the induction argument we use the following convention. For a vector \(y\) of length \(m\), the quantities \(u_r(y)\), \(v_{r,s}(y)\), and \(A_m^{(r,s)}(y)\) are defined using \(e_r^{(i)}(y)=0\) whenever \(r<0\) or \(r>m-1\). The scalar \(s_{r,s}(y)\) is defined by the same formula as above, with \(e_j(y)=0\) for \(j>m\). Thus, if either \(r>m-1\) or \(s>m-1\), then \(v_{r,s}(y)=0\) and \(A_m^{(r,s)}(y)=0\), while \(s_{r,s}(y)\ge0\) by the scalar positivity lemma. 

We  prove that (\ref{result-extended}) holds
by resorting to 
the principal-minor criterion. Since, by Lemma \ref{detcalH=0}
\[
\det\mathcal H_n^{(p,q)}(x)=0,
\]
it remains to show that every proper principal minor is nonnegative.
To this end, we index the  first row and column of \(\mathcal H_n^{(p,q)}\) by
\(0\).

First consider the principal submatrices of  \(\mathcal H_n^{(p,q)}\) containing the index \(0\).
Consider one such a submatrix indexed by
\[
\{0\}\cup I,
\qquad
I\subsetneq\{1,\dots,n\}.
\]
If \(I=\varnothing\), the corresponding principal submatrix is the \(1\times1\) matrix
$s_{p,q}(x)$, which is positive semidefinite by Lemma 3.1. Hence we may assume \(I\ne\varnothing\).

Let
\[
J:=\{1,\dots,n\}\setminus I.
\]
Using the decompositions
\[
e_p^{(i)}(x)
=
\sum_{a=0}^{p}e_a(x_J)e_{p-a}^{(i)}(x_I),
\qquad
e_q^{(i)}(x)
=
\sum_{b=0}^{q}e_b(x_J)e_{q-b}^{(i)}(x_I),
\]
we get
\[
v_{p,q}(x)[I]
=
\sum_{a=0}^{p}
\sum_{b=0}^{q}
e_a(x_J)e_b(x_J)
v_{p-a,q-b}(x_I),
\]
and
\[
A_n^{(p,q)}(x)[I]
=
\sum_{a=0}^{p}
\sum_{b=0}^{q}
e_a(x_J)e_b(x_J)
A_{|I|}^{(p-a,q-b)}(x_I),
\]
where $v_{p,q}(x)[I]$ 
and
$
A_n^{(p,q)}(x)[I]$ denote, respectively, the restriction of $v_{p,q}(x)$ 
and
$
A_n^{(p,q)}(x)$ to the set if indexes $I$.
Consequently,
\beq\label{eq:block-decomposition}
\mathcal H_n^{(p,q)}(x)[\{0\}\cup I]
=
\sum_{a=0}^{p}
\sum_{b=0}^{q}
e_a(x_J)e_b(x_J)
\mathcal H_{|I|}^{(p-a,q-b)}(x_I)
+
\begin{bmatrix}
R_{p,q}(x_I,x_J)&0\\
0&0
\end{bmatrix},
\eeq
where
\begin{equation}\label{eq:R-definition}
  R_{p,q}(x_I,x_J)
  :=
  s_{p,q}(x)
  -
  \sum_{a=0}^{p}\sum_{b=0}^{q}
  e_a(x_J)e_b(x_J)s_{p-a,q-b}(x_I).
\end{equation}
is the difference between \(s_{p,q}(x)\) and the corresponding sum of the
top-left entries.
Let \(\mu:=\min\{p,q\}\). We now show that
\begin{equation}\label{eq:R-explicit}
  R_{p,q}(x_I,x_J)
  =
  \sum_{d=0}^{\mu}
  \binom{p+q-2d}{\mu-d}
  \sum_{\substack{
    L\subseteq\{1,\dots,n\},\ T\subseteq J\\
    L\cap T=\varnothing\\
    |L|=d,\ |T|=p+q+1-2d
  }}
  x_L^2x_T.
\end{equation}
Equivalently, $R_{p,q}$ consists precisely of those monomials, in the positive
expansion \eqref{posexspq} of $s_{p,q}(x)$, for which every variable appearing to the first
power belongs to $J$. This clearly implies $
R_{p,q}(x_I,x_J)\geq 0.
$
From now on, we  assume \(p\le q\) so that $\mu=p$. By symmetry, this assumption can be done without loss of generality.
To prove (\ref{eq:R-explicit}), 
consider a monomial $x_L^2x_T$ in the positive expansion
\eqref{posexspq} of $s_{p,q}(x)$, so that
\[
  L,T\subseteq\{1,\dots,n\},
  \qquad
  L\cap T=\varnothing,
  \qquad
  |L|=d,
  \qquad
  |T|=p+q+1-2d.
\]
Its coefficient in $s_{p,q}(x)$ is
\begin{equation}\label{eq:full-coefficient}
  \binom{p+q-2d}{p-d}.
\end{equation}
We compare this with its coefficient in the  term
\[
  \sum_{a=0}^{p}\sum_{b=0}^{q}
  e_a(x_J)e_b(x_J)s_{p-a,q-b}(x_I).
\]

First suppose that $T\cap I\ne\varnothing$. Let
\[
  t_I:=|T\cap I|,
  \qquad
  t_J:=|T\cap J|,
  \qquad
  d_J:=|L\cap J|.
\]
In a contribution from the term indexed by $(a,b)$, the two factors
$e_a(x_J)$ and $e_b(x_J)$ must contain every variable in $L\cap J$, and they
must split the variables in $T\cap J$. If $h$ variables of $T\cap J$ are placed
in $e_a(x_J)$, then
\[
  a=d_J+h,
  \qquad
  b=d_J+t_J-h,
  \qquad
  0\le h\le t_J,
\]
and there are $\binom{t_J}{h}$ such choices.

For this fixed $h$, the remaining part of the monomial must be supplied by
$s_{p-a,q-b}(x_I)$. Its coefficient is
\[
  \binom{t_I-1}{p-d-h}.
\]
Indeed, the upper index is $t_I-1$, while the two possible lower indices are
$p-d-h$ and $t_I-1-(p-d-h)$; the two give the same binomial coefficient.
Therefore the coefficient of $x_L^2x_T$ in the convolution term is
\beq\label{sumcanc}
  \sum_{h=0}^{t_J}
  \binom{t_J}{h}
  \binom{t_I-1}{p-d-h}=\binom{t_I+t_J-1}{p-d}=
  \binom{p+q-2d}{p-d}
\eeq
where we used the Vandermonde's identity and the fact that
$t_I+t_J=|T|=p+q+1-2d$.
The right-hand side of (\ref{sumcanc}) is exactly the coefficient \eqref{eq:full-coefficient}. Hence every monomial with $T\cap I\ne\varnothing$ cancels in the difference
\eqref{eq:R-definition}.

Now suppose that $T\subseteq J$. Then the convolution term cannot produce
$x_L^2x_T$. Indeed, each nonzero monomial in a positive expansion of some
$s_{r,s}(x_I)$ has at least one variable from $I$ appearing to the first power,
since its set of variables appearing to the first power has cardinality
$r+s+1-2\delta$ with $\delta\le \min\{r,s\}$.
Thus the monomial $x_L^2x_T$ survives in $R_{p,q}(x_I,x_J)$ with its full
coefficient \eqref{eq:full-coefficient}.
Consequently, the only surviving monomials are exactly those with $T\subseteq J$.
This proves \eqref{eq:R-explicit}. 
Since all the coefficients in
\eqref{eq:R-explicit} are binomial coefficients and all the monomials $x_L^2x_T$
are nonnegative for $x_i>0$, we conclude that
\[
  R_{p,q}(x_I,x_J)\ge0.
\]

Therefore, in \eqref{eq:block-decomposition}, all coefficients
$e_a(x_J)e_b(x_J)$ are nonnegative, the matrices
$\mathcal H_{|I|}^{(p-a,q-b)}(x_I)$ are positive semidefinite by the induction
hypothesis, and the scalar block
containing $R_{p,q}(x_I,x_J)$ is positive semidefinite. Hence
\[
  \mathcal H_n^{(p,q)}(x)[\{0\}\cup I]\succeq0.
\]

In conclusion, every proper principal minor containing the index \(0\) is nonnegative.

It remains to prove that every proper principal submatrix of
$\mathcal H_n^{(p,q)}(x)$
not containing the index \(0\) is positive semidefinite.
Such a submatrix is 
a
principal submatrix of \(A_n^{(p,q)}(x)\), indexed by some
\[
I\subseteq\{1,\dots,n\}.
\]
If \(I\subsetneq\{1,\dots,n\}\), put
\[
J:=\{1,\dots,n\}\setminus I.
\]
For \(i\in I\),
\[
e_p^{(i)}(x)
=
\sum_{a=0}^{p}e_a(x_J)e_{p-a}^{(i)}(x_I),
\]
and similarly
\[
e_q^{(i)}(x)
=
\sum_{b=0}^{q}e_b(x_J)e_{q-b}^{(i)}(x_I).
\]
Therefore
\[
A_n^{(p,q)}(x)[I]
=
\sum_{a=0}^{p}
\sum_{b=0}^{q}
e_a(x_J)e_b(x_J)
A_{|I|}^{(p-a,q-b)}(x_I),
\]
where terms with negative indices or indices larger than \(|I|-1\) are
interpreted as zero, and we use the symmetry
$
A_{|I|}^{(r,s)}=A_{|I|}^{(s,r)}.
$
By the induction hypothesis, every nonzero matrix on the right-hand side is
positive semidefinite. Since all coefficients are nonnegative, we get
\[
A_n^{(p,q)}(x)[I]\succeq0.
\]
Thus every such principal minor is nonnegative.

\newpage

It remains to treat the principal submatrix indexed by the full lower set
\(\{1,\dots,n\}\), namely \(A_n^{(p,q)}(x)\).

If \(\gamma_{p,q}=0\), then \(A_n^{(p,q)}(x)w=0\). Since \(w\ne0\), it follows that
\[
\det A_n^{(p,q)}(x)=0.
\]

Assume now that \(\gamma_{p,q}\ne0\). If
\[
\operatorname{rank}\mathcal H_n^{(p,q)}(x)\le n-1,
\]
then all \(n\times n\) minors of \(\mathcal H_n^{(p,q)}(x)\) vanish, and in particular
\[
\det A_n^{(p,q)}(x)=0.
\]
Otherwise \(\operatorname{rank}\mathcal H_n^{(p,q)}(x)=n\). Since
\[
\mathcal H_n^{(p,q)}(x)
\begin{bmatrix}
-\gamma_{p,q}\\
w
\end{bmatrix}
=0,
\]
the kernel is one-dimensional and the adjugate has the form
\[
\operatorname{adj}\mathcal H_n^{(p,q)}(x)
=
c
\begin{bmatrix}
-\gamma_{p,q}\\
w
\end{bmatrix}
\begin{bmatrix}
-\gamma_{p,q}\\
w
\end{bmatrix}^{\top}
\]
for some real number \(c\). Choose \(j\) such that \(w_j\ne0\). The principal cofactor obtained by deleting the \(j\)-th lower row and column is a proper principal minor containing the distinguished index \(0\), hence it is nonnegative by the previous part of the proof. Therefore \(c w_j^2\ge0\), so \(c\ge0\). The cofactor corresponding to the distinguished index \(0\) is \(\det A_n^{(p,q)}(x)\), and hence
\[
\det A_n^{(p,q)}(x)=c\gamma_{p,q}^2\ge0.
\]

\newpage

It remains
 to treat the submatrix indexed by the full
set \(I=\{1,\dots,n\}\), namely \(A_n^{(p,q)}(x)\) itself.

If \(\gamma_{p,q}=0\), then (\ref{kerofcalH}) yields
$
A_n^{(p,q)}(x)w=0.
$
Since \(w\ne0\), this gives
\beq\label{detA=0}
\det A_n^{(p,q)}(x)=0.
\eeq
Moreover, if 
\(\operatorname{rank}\mathcal H_n^{(p,q)}(x)\le n-1\), then every \(n\times n\)
minor of \(\mathcal H_n^{(p,q)}(x)\) is zero, so that
(\ref{detA=0}) holds.

We now consider the case 
of \(\gamma_{p,q}\ne0\) and
\(\operatorname{rank}\mathcal H_n^{(p,q)}(x)=n\).
In this case, 
the kernel of $H_n^{(p,q)}(x)$ is a $1$-dimensional space generated by the vector
\[
\begin{bmatrix}
-\gamma_{p,q}\\
w
\end{bmatrix}
\]
Hence  the adjugate of
\(\mathcal H_n^{(p,q)}(x)\) has rank one and is given by
\[
\operatorname{adj}\mathcal H_n^{(p,q)}(x)
=
c
\begin{bmatrix}
-\gamma_{p,q}\\
w
\end{bmatrix}
\begin{bmatrix}
-\gamma_{p,q}\\
w
\end{bmatrix}^{\top}
\]
for some real number \(c\).

Choose \(j\) such that \(w_j\ne0\). The principal minor obtained by deleting
the \(j\)-th row and column among the lower \(n\) rows and columns is a proper
principal minor containing the  index \(0\) and, as shown in the first part of the proof, such
principal minors are nonnegative. Therefore
$
c w_j^2\ge0,
$
so that
$
c\ge0.
$
The cofactor corresponding to the index \(0\) is exactly
\[
\det A_n^{(p,q)}(x).
\]
Hence
\[
\det A_n^{(p,q)}(x)
=
c\gamma_{p,q}^2
\ge0.
\]

We have shown that every proper principal minor of
\(\mathcal H_n^{(p,q)}(x)\) is nonnegative, and we already know that
(\ref{detH=0}) holds.
Therefore all principal minors of \(\mathcal H_n^{(p,q)}(x)\) are nonnegative.
By the principal-minor criterion, (\ref{result-extended}) holds. This completes the induction and the proof.
\end{proof}

\section{Conclusion}

We have proved that the Cauchy-like matrix
\[
C_{ij}=\frac{r_i(k,l)+r_j(k,l)}{\lambda_i+\lambda_j},
\qquad
r_i(k,l)=\frac{\sigma_{n-k}^{(i)}}{\sigma_{n-l}^{(i)}},
\]
is positive semidefinite for all positive distinct nodes
$\lambda_1,\ldots,\lambda_n$ and for every $1\le k,l\le n$.  The proof is based
on three ingredients: a diagonal-congruence reduction to the matrices
$A_n^{(p,q)}(x)$ in reciprocal variables; a Cauchy-kernel identity producing a
nonzero kernel vector for the augmented matrices
$\mathcal H_n^{(p,q)}(x)$; and an induction argument using principal minors.  A
notable feature of the proof is the explicit positive expansion of the scalar
$s_{p,q}$, together with the corresponding positive expansion of the remainder
$R_{p,q}(x_I,x_J)$ that appears when restricting to a proper principal
submatrix.  This coefficientwise description is what makes the induction close
for the full two-parameter family.

\bibliographystyle{plain}

\bibliography{Cauchy-like-matrices}

\end{document}